\documentclass[11pt]{article}

\usepackage{amsmath, amsthm}
\usepackage{amssymb}
\usepackage{stmaryrd}
\usepackage{graphicx}
\usepackage{setspace}
\usepackage{textcomp}
\usepackage{esvect}
\usepackage{arydshln}
\usepackage[all]{xy}
\usepackage{etex}
\usepackage[dvipsnames]{xcolor}  
\usepackage{epsfig}
\usepackage{latexsym}
\usepackage{amsthm}
\usepackage{mathrsfs}
\usepackage[colorlinks,citecolor=blue]{hyperref}
\usepackage{tikz-cd}
\usetikzlibrary{graphs}
\usepackage{pgfplots}

\newtheorem{theorem}{Theorem}[section]
\newtheorem{lemma}[theorem]{Lemma}
\newtheorem{proposition}[theorem]{Proposition}
\newtheorem{corollary}[theorem]{Corollary}

\theoremstyle{definition}
\newtheorem{definition}[theorem]{Definition}

\newcommand{\Fisg}{\text{\textnormal{Fisg}}}
\newcommand{\Iisg}{\text{\textnormal{Iisg}}}
\newcommand{\Pisg}{\text{\textnormal{Pisg}}}
\newcommand{\Tisg}{\text{\textnormal{Tisg}}}
\newcommand{\Dom}{\text{\textnormal{Dom}}}
\newcommand{\Img}{\text{\textnormal{Im}}}

\input xy
\xyoption{all}

\begin{document}
 \title{\sffamily Graphs and Their Associated Inverse Semigroups}

  \author{
    \textsc{T. Chih and D. Plessas} \\[0.25em]
    }

\maketitle

\begin{abstract}
Directed graphs have long been used to gain understanding of the structure of semigroups, and recently the structure of directed graph semigroups has been investigated resulting in a characterization theorem and an analog of Fruct's Theorem. We investigate four inverse semigroups defined over undirected graphs constructed from the notions of subgraph, vertex induced subgraph, rooted tree induced subgraph, and rooted path induced subgraph. We characterize the structure of the semilattice of idempotents and lattice of ideals of these four inverse semigroups. Finally, we prove a characterization theorem that states that every graph has a unique associated inverse semigroup up to isomorphism.
\end{abstract}

\section{Introduction}\label{Intro}

We will follow the notations of \cite{Bondy} for graph theory, \cite{Stanley} for semilattices and lattices, and \cite{SGT} for inverse semigroups. We will only consider finite undirected graphs, but they are allowed to have multiple edges and loops. We allow $\emptyset$ to be considered a graph without vertices or edges and the empty map $\mu_0:\emptyset \to \emptyset$ to be a valid graph isomorphism.\par
Much of the theory linking semigroups to graphs has been in the guise of directed graphs \cite{Ash, Margolis, Kelarev, Jones, Sieben}. However, undirected graphs have rich internal symmetries for which groups are too coarse an algebraic structure to distinguish. This has lead to notions of distinguishing number \cite{Albertson} and fixing number \cite{Gibbons}. Furthermore, local symmetry in the form of subgraph embeddings has been used famously by Lov{\'a}sz to solve the edge reconstruction conjecture \cite{Harary} for graphs with $n$ vertices and $m$ edges where $m\geq 1/2{n \choose 2}$ \cite{Lovasz}. The algebra of studying local symmetry is an inverse semigroup. This leads us to investigate inverse semigroups on undirected graphs.\par

In section \ref{ISGs}, we begin by defining inverse semigroups associated to undirected graphs to correspond to the ideas of subgraph symmetry, vertex induced subgraph symmetry, tree induced subgraph symmetry, and path induced subgraph symmetry. These four inverse semigroups are linked to three famous conjectures in graph theory: the edge reconstruction conjecture \cite{Harary}, the vertex reconstruction conjecture \cite{BondyS}, and the Lov{\'a}sz conjecture \cite{LovaszC} which states that every vertex transitive graph contains a hamiltonian path.  These inverse semigroups  are graph analogues of the inverse semigroup of sets \cite{WPT} with a necessary restriction of partial monomorphism to partial isomorphism \cite{IndSym}.\par

In section \ref{SemilatticeS} we characterize the semilattice of idempotents for these four inverse semigroups and their relation to the subgraph structure of the associated graph. In section \ref{Ideals} we characterize the ideals of these four inverse semigroups and characterize their ideal lattices. Finally, in section \ref{Characterization} we prove that the inverse semigroup corresponding to all subgraphs of a graph uniquely determines that graph up to isomorphism.\par 

\section{Inverse Semigroups Constructed from Graph Symmetry}\label{ISGs}

We will start with the most general inverse semigroup associated to all subgraphs of a graph.

\begin{definition}
Let $G$ be a graph. We define $\Fisg(G)$ to be the collection of all graph isomorphisms $\varphi:H\to J$ where $H$ and $J$ are subgraphs of $G$.
\end{definition}

We then define composition. For $\psi, \varphi \in \Fisg(G)$ we define $\psi\circ\varphi:\varphi^{-1}(\Dom(\psi))\to \psi(\text{Im}(\varphi))$ to be $\psi\circ\varphi=\psi\circ \varphi|_{\varphi^{-1}(\text{Dom}(\psi))}$ and notice that $\psi\circ\varphi$ is an isomorphism of subgraphs.\par

The composition of $\Fisg(G)$ is associative, and for any subgraph isomorphism $\varphi$, $\varphi\circ\varphi^{-1}\circ\varphi=\varphi$. Hence $\Fisg(G)$ forms an inverse semigroup under composition. As we will see in section \ref{Ideals}, $\Fisg(G)-\text{Aut}(G)$ is an ideal of $\Fisg(G)$. When $G$ is connected, a set of generators of this ideal is the set of identity isomorphisms of edge deleted subgraphs of $G$, highlighting a strong link to the edge reconstruction conjecture.\par

We get an analogous connection for the vertex reconstruction conjecture if we instead consider an inverse semigroup associated to vertex induced subgraphs of $G$.

\begin{definition}
Let $G$ be a graph. We define $\Iisg(G)$ to be the collection of all graph isomorphisms $\varphi:H\to J$ where $H$ and $J$ are vertex induced subgraphs of $G$.
\end{definition}

We then define composition the same as for $\Fisg(G)$, and note that the intersection of two vertex induced subgraphs is a vertex induced subgraph. Hence $\Iisg(G)$ is also an inverse semigroup under composition, and the ideal $\Iisg(G)-\text{Aut}(G)$ has a set of generators of identity isomorphisms of vertex deleted subgraphs of $G$.\par

We now move to two inverse semigroups who are linked to the Lov{\'a}sz conjecture. We would like to consider path induced subgraphs graph, as a graph $G$ has a hamiltonian path if and only if $G$ is a path induced subgraph of $G$. However, there is no natural well defined way to intersect two paths and be guaranteed a path.\par
As the Lov{\'a}sz conjecture relates to vertex transitive graphs, we could consider rooting the paths at any vertex. However, we still run into a problem where the intersection of path induced subgraphs is a non-path tree. For an example of this, given any vertex root of the Petersen graph, two rooted $5$-cycles generated by rooted paths that share the two edges incident to the root in their cycles will have an intersection of a vertex rooted non-path tree.\par

\begin{figure}[h]
\includegraphics[scale=.38]{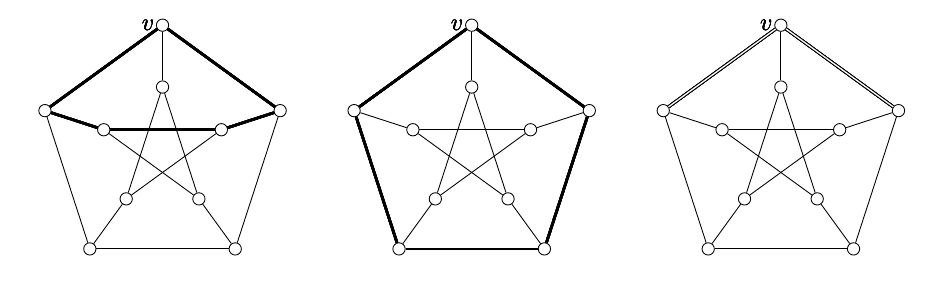}
\caption{Two $v$-rooted path induced cycles of the Petersen Graph have a $v$-rooted non-path tree intersection.}
\end{figure}

So we must either consider vertex rooted tree induced subgraphs or we must consider vertex rooted path induced subgraphs paired with their vertex rooted path. We will consider both, starting with the inverse semigroup associated to vertex rooted tree induced subgraphs.

\begin{definition}
Let $G$ be a graph and $v\in V(G)$. We define $v$-rooted tree induced subgraph $H$ of $G$ to be such that $H=G[V(T)]$ for some tree $T$ of $G$ rooted at $v$.
\end{definition}

\begin{definition}
Let $G$ be a graph and $v\in V(G)$. We define $\Tisg(G,v)$ to be the collection of graph isomorphisms $\varphi:H\to J$ where $\varphi(v)=v$ and $H$ and $J$ are $v$-rooted tree induced subgraphs of $G$. 
\end{definition}

Defining composition is now a bit more technical. For $\varphi, \psi \in \Tisg(G,v)$ we define $\psi\circ\varphi:\varphi^{-1}(C_{\psi\circ\varphi})\to \psi(C_{\psi\circ\varphi})$ where $C_{\psi\circ\varphi}=G[V(C_{\psi\circ\varphi})]$ and
$V(C_{\psi\circ\varphi})$ contains all vertices $u\in V(G)$ such that there is a $vu$ path contained in both $\text{Im}(\varphi)$ and $\text{Dom}(\psi)$, and $\psi\circ\varphi=\psi\circ\varphi{}|_{\varphi^{-1}(C_{\psi\circ\varphi})}$.

\begin{proposition}\label{TisgISG}
Let $G$ be a graph and $v\in V(G)$, then $\Tisg(G,v)$ is an inverse semigroup.
\end{proposition}
\begin{proof}
Let $\varphi, \psi, \alpha \in \Tisg(G,v)$. We first note that $\psi \circ \varphi$ is an isomorphism between $v$-rooted tree induced subgraphs of $G$. For if $x,y\in V(C_{\psi\circ\varphi})$ then there is a $v,x$-path contained in both $\Img(\varphi)$ and $\Dom(\psi)$, and similarly for a $v,y$-path. Hence there is a $x,y$-path contained in $\Img(\varphi)$ and $\Dom(\psi)$ and $G[ V(C_{\psi\circ\varphi})]$ is connected. Then it contains a spanning tree which we root at $v$, and as $\psi$ and $\varphi$ preserve $v$, the inverse image and image of $G[ V(C_{\psi\circ\varphi})]$ are also $v$-rooted tree induced subgraphs of $G$. \par
So it now suffices to show composition is associative. Consider $\alpha\circ(\psi\circ\varphi):(\psi\circ\varphi)^{-1}(C_{\alpha\circ(\psi\circ\varphi)})\to\alpha(C_{\alpha\circ(\psi\circ\varphi)})$. Let $u\in V((\psi\circ\varphi)^{-1}(C_{\alpha\circ(\psi\circ\varphi)}))$, then $\psi\circ\varphi(u)\in V(C_{\alpha\circ(\psi\circ\varphi)})$ and there is a $v,\psi(\varphi(u))$-path in both $\Img(\psi\circ\varphi)$ and $\Dom(\alpha)$. As $\Img(\psi\circ\varphi)=\psi(C_{\psi\circ\varphi})$, there is a $\psi^{-1}(v),\psi^{-1}(\psi(\varphi(u)))=v,\varphi(u)$-path in $C_{\psi\circ\varphi}$. Then there is a $v,\varphi(u)$-path in both $\Img(\varphi)$ and $\Dom(\psi)$. Then the $\psi(v),\psi(\varphi(u))=v,\psi(\varphi(u))$-path is a path in both $\Img(\psi)$ and $\Dom(\alpha)$. Hence $\psi(\varphi(u))\in V(C_{\alpha\circ\psi})$, and the $v,\varphi(u)$-path is both a path in $\Img(\varphi)$ and $\Dom(\alpha\circ\psi)$. Therefore $\varphi(u)\in V(C_{(\alpha\circ\psi)\circ\varphi)})$ and $u\in V(\varphi^{-1}(C_{(\alpha\circ\psi)\circ\varphi)}))$. It is similar to show the reverse containment.\par  Hence $\varphi^{-1}(V(C_{(\alpha\circ\psi)\circ\varphi)}))=V((\psi\circ\varphi)^{-1}(C_{\alpha\circ(\psi\circ\varphi)}))$, and $\alpha\circ(\psi\circ\varphi)$ and $(\alpha\circ \psi)\circ \varphi$ have the same domain. It is similar to show they have the same codomain. Finally, as they have the same domain, codomain, and are restrictions across the composition of isomorphisms, $\alpha\circ(\psi\circ\varphi)=(\alpha\circ\psi)\circ\varphi$.
\end{proof}

We now consider the case where we distinguish the vertex rooted path that induces the subgraph.

\begin{definition}
Let $G$ be a graph and $v\in V(G)$. We define $v$-rooted path induced subgraph $H$ of $G$ to be such that $H=G[V(P_H)]$ for some path $P_H$ of $G$ rooted at $v$. We call the pair $(H,P_H)$ the $v$-rooted path pair.
\end{definition}

\begin{definition}
Let $G$ be a graph and $v\in V(G)$. We define $\Pisg(G,v)$ to be the collection of graph isomorphisms of $v$-rooted path pairs, that is subgraph isomorphisms $(\varphi,\varphi{|}_{P_H}):(H,P_H)\to (J, P_J)$ where $\varphi:H\to J$ is a subgraph isomorphism, $\varphi(v)=v$, and $\varphi(P_H)=P_J$. When it does not cause ambiguity, we will refer to $(\varphi,\varphi{|}_{P_H})$ as $\varphi$.
\end{definition}

Composition will be defined similarly to that of $\Tisg(G,v)$ taking into account the distinguished path. Let $\varphi, \psi \in \Pisg(G,v)$. Then $\varphi:(H, P_H) \to (J,P_J)$ and $\psi:(K,P_K) \to (M,P_M)$. Define $P_{\psi\circ \varphi}$ to be the longest $v$-rooted path in common to both $P_J$ and $P_K$. Define $C_{\psi \circ \varphi}=G[V(P_{\psi\circ\varphi})]$. We then define $\psi\circ \varphi:(\varphi^{-1}(C_{\psi\circ\varphi}),\varphi^{-1}(P_{\psi\circ\varphi}))\to(\psi(C_{\psi\circ\varphi}),\psi(P_{\psi\circ\varphi}))$ to be $\psi\circ\varphi{|}_{\varphi^{-1}(C_{\psi\circ\varphi})}$. Then by noting that a vertex $u\in V(C_{\psi\circ\varphi})$ is on a $v,u$-path contained in both $P_J$ and $P_K$, namely $P_{\psi\circ\varphi}$, the following proposition follows from a similar proof to Proposition \ref{TisgISG}.

\begin{proposition}\label{PisgISG}
Let $G$ be a graph and $v\in V(G)$. Then $\Pisg(G,v)$ is an inverse semigroup.
\end{proposition}


\section{Semilattice Structures}\label{SemilatticeS}
In this section we concern ourselves with the semilattice structure formed by the idempotents of our inverse semigroups. The following lemma will be useful in determining the structure of these semilattices.
\begin{lemma}\label{idemps}
Let $G$ be a graph, then $e\in\Fisg(G)$ is idempotent if and only if there exists a subgraph $H$ of $G$ with $e=id_{H}$, the identity automorphism.
\end{lemma}
\begin{proof}
($\Leftarrow$) For every subgraph $H$, $id_{H}$ is idempotent.\\
($\Rightarrow$) Let $e\in\Fisg(G)$ be an idempotent. Then as $ee=e$, $\Dom(e)=\Img(e)$ and $e$ is an automorphism of a subgraph $H=\Dom(e)$. As \text{Aut}$(H)$ is a group, and the only idempotent of a group is the identity, $e=id_{H}$.
\end{proof}

As for any graph $G$ and $v\in V(G)$, $\Iisg(G)$, $\Tisg(G,v)$ are subsemigroups of $\Fisg(G)$, we note the idempotents are the identities of vertex induced subgraphs and $v$-rooted tree induced subgraphs for these subsemigroups respectively. Using this lemma we now characterize the idempotent semilattice structure of $\Fisg(G), \Iisg(G),$ and $\Tisg(G,v)$.

\begin{theorem}\label{FisgLattice}
Given a graph $G$, the semilattice of idempotents of $\Fisg(G)$ form a bi-Heyting Algebra.
\end{theorem}
\begin{proof}
Let $e$ and $f$ be idempotents of $\Fisg(G)$ with $e\leq f$. Thus $e=ef$. By Lemma \ref{idemps}, $e=id_{H}$ for some subgraph $H$ of $G$ and $f=id_{J}$ for some subgraph $J$ of $G$. Hence $id_{H}=id_{H}\circ id_{J}$. By definition of composition in $\Fisg(G)$ this means that $H$ is a subgraph of $J$. The converse trivially holds. Thus we have order preserving isomorphism between the semilattice of idempotents of $\Fisg(G)$ and that of subgraphs of $G$ ordered by inclusion. By \cite{Lawvere, Reyes} the semilattice of subgraphs of a graph form a bi-Heyting Algebra.
\end{proof}

We note that as there are inverse semigroups whose semilattice of idempotents are non-graded lattices, there is no hope of establishing a Fruct style theorem for inverse semigroups using $\Fisg(G)$. The next two theorems establish graded lattice structures for $\Iisg(G)$ and $\Tisg(G,v)$.

\begin{theorem}\label{IisgLattice}
Given a graph $G$, the semilattice of idempotents of $\Iisg(G)$ form a Boolean Algebra.
\end{theorem}
\begin{proof} Similarly to the proof of Theorem \ref{FisgLattice}, as the idempotents of $\Iisg(G)$ are the identities of induced subgraphs of $G$, the semilattice of idempotents of $\Iisg(G)$ is order isomorphic to the semilattice of vertex induced subgraphs of $G$ ordered by inclusion. However, this semilattice is order isomorphic to the Boolean Algebra of subsets of $V(G)$ ordered by inclusion, for given $X, Y$ subsets of $V(G)$, $X\subseteq Y$ if and only if $G[X]$ is a subgraph of $G[Y]$. 
\end{proof}

\begin{theorem}\label{TisgLattice}
Given a graph $G$ and $v\in V(G)$, the semilattice of idempotents of $\Tisg(G,v)$ form a graded lattice.
\end{theorem}
\begin{proof}
Similarly to the proof of Theorem \ref{FisgLattice}, as the idempotents of $\Tisg(G,v)$ are the identities of $v$-rooted tree induced subgraphs of $G$, the semilattice of idempotents of $\Tisg(G,v)$ is order isomorphic to the semilattice of $v$-rooted tree induced subgraphs of $G$ ordered by inclusion.\par
 Let $H_1$ and $H_2$ be $v$-rooted tree induced subgraphs of $G$. We define $H_1\wedge H_2=G[V(H_1\wedge H_2)]$ where $V(H_1\wedge H_2)$ is the set of vertices such that there is a $v,u$-path contained both in $H_1$ and $H_2$. We note that $H_1\wedge H_2$ is a $v$-rooted tree induced subgraph of $G$ by a similar argument to the proof of Proposition \ref{TisgISG}. So it suffices to show it is meet of $H_1$ and $H_2$.\par
Now let $K$ be a $v$-rooted tree induced subgraph of $G$ such that $K$ is a subgraph of both $H_1$ and $H_2$. Let $u\in V(K)$. As $K$ contains a $v$-rooted tree that spans it, there exists a $v,u$-path contained in $K$. As $K$ is a subgraph of $H_1$ and $H_2$, this $v,u$-path is contained in both $H_1$ and $H_2$. Hence $u\in V(H_1\wedge H_2)$, and $V(K)\subseteq V(H_1\wedge H_2)$. Then $K=G[V(K)]$ is a subgraph of $H_1\wedge H_2$.\par
We now define $H_1\vee H_2=G[V(H_1)\cup V(H_2)]$. For any two vertices $x,y\in V(H_1)\cup V(H_2)$ we have a path from $v$ to $x$ in either $H_1$ or $H_2$ and similarly for $v$ to $y$. Thus $H_1\vee H_2$ is connected and contains a spanning tree. Rooting this tree at $v$ yields $H_1 \vee H_2$ as a $v$-rooted tree induced subgraph of $G$.\par
Now let $J$ be a $v$-rooted tree induced subgraph of $G$ that contains $H_1$ and $H_2$ as subgraphs, and let $u\in V(H_1\vee H_2)=V(H_1)\cup V(H_2)$. Then as both $H_1$ and $H_2$ are subgraphs of $J$, $u\in V(J)$. Then as $V(H_1\vee H_2)\subseteq V(J)$, $H_1\vee H_2$ is a subgraph of $J$.\par
As is $G$ is finite, the lattice is bounded. For the component of $G$ containing $v$, all $v$-rooted spanning trees have the same number of edges. Hence every chain of the lattice has the same length, and the lattice is graded.
\end{proof}

Finally, we see that $\Pisg(G,v)$ will not be a lattice except for a specific class of rooted graphs.

\begin{proposition}\label{PisgSemilattice}
Let $G$ be a graph and $v\in V(G)$, then the idempotents of $\Pisg(G,v)$ form a lattice if and only if $G$ is a $v$-rooted path.
\end{proposition}
\begin{proof}
$(\Rightarrow)$ Similarly to the proof of Theorem \ref{FisgLattice}, the idempotents of $\Pisg(G,v)$ are the identities of $v$-rooted path pairs, and the idempotent order is order isomorphic to the inclusion order of $v$-rooted path pairs. So suppose that $G$ is not a $v$-rooted path. There there are two distinct $v$-rooted paths $P_H$ and $P_J$ in $G$ with $v$-rooted path pairs $(H,P_H)$ and $(J,P_J)$. As the only $v$-rooted subgraph that can contain both $P_H$ and $P_J$ contains either a non-path tree or a cycle, $(H,P_H)\vee (J,P_J)$ does not exist. Hence $\Pisg(G,v)$ is not a lattice.\par
$(\Leftarrow)$  As the semilattice of $v$-rooted path pairs of a $v$-rooted path is a chain, it is a lattice.
\end{proof}

\section{Ideals and Ideal Lattice Structures}\label{Ideals}

In this section, we study the ideals of our graph inverse semigroups and the lattice structure of these ideals.

\begin{definition}
Given a semigroup $S$, a \textit{two sided ideal} (or simply \textit{ideal}) of $S$, $I$ is a set $I\subseteq S$ where $SI, IS\subseteq I$.  \cite{SGT}
\end{definition}

This is of particular note since ideals of a semigroup induce an equivalence relationship which leads to the construction of a quotient semigroup:

\begin{theorem}[Rees Factor Theorem]
Let $S$ be a semigroup and let $I$ be an ideal of $S$.  Then define the relation $\sim_I\subseteq S\times S$ where $a\sim_I b$ if and only if $a,b\in I$ or $a=b$.  It follows that:

\begin{enumerate}
\item $\sim$ is an equivalence relationship.
\item $S/\sim_I$ is a well defined factor semigroup.
\end{enumerate}\cite{Rees}
\end{theorem}

Since each ideal $I$ of $S$ gives rise to a factor semigroup, and consequently a kernel of a semigroup homomorphism, one naturally wishes to classify all such ideals for the inverse semigroups of graphs and consider their structure.\\

We begin with a utility lemma regarding ideals.

\begin{lemma}\label{lemideal}
Let $S$ be a semigroup and $\{I_i\}$ be a collection of ideals of $S$.  Then $\bigcap I_i$ and $\bigcup I_i$ are ideals of $S$.\cite{SGT}
\end{lemma}

Then, the following results with proof for $\Fisg(G)$ are given.

\begin{definition}\label{defidealgenf}
Given a graph $G$ and $a\in \Fisg(G)$ where $a:H\to K$, we call $\langle a \rangle:=\{\varphi:L\to M\ |\ \text{$L, M$ are isomorphic to a  subgraph of $H$}\}$ the \textit{ideal generated by $a$}.
\end{definition}

We note that since $\mu_0$, the empty map, is in $\Fisg(G)$, $\mu_0\circ a=a\circ\mu_0=\mu_0\in \langle a \rangle$.  To earn its name, we prove that $\langle a \rangle$ is in fact an ideal.

\begin{proposition}\label{propisidealf}
Given a graph $G$ and $a\in \Fisg(G)$ where $a:H\to K$, it follows that $\langle a \rangle$ is an ideal.
\end{proposition}
\begin{proof}
Let $\alpha\in \Fisg(G)$.  We consider $\alpha\circ a$.  If $\text{Im}(a)\cap \text{Dom}(\alpha)=\emptyset$ then $\alpha\circ a$ is the empty map and is in $\langle a \rangle$.  Otherwise $\alpha\circ a=\alpha\circ a|_{a^{-1}(\text{Dom}(\alpha))}$.  Notice that $\text{Im}(a)\cap \text{Dom}(\alpha)$ is a subgraph of $\text{Im}(a)$.  Since $\text{Im}(a)$ is isomorphic to  $H$, $\text{Im}(a)\cap \text{Dom}(\alpha)$ is isomorphic to a subgraph of $H$ and $\alpha\circ a$ is an isomorphism of  subgraphs isomorphic to a subgraph of $H$.

Similarly, consider $a\circ \alpha$.  If $\text{Im}(\alpha)\cap \text{Dom}(a)=\emptyset$ then $a\circ\alpha$ is the empty map and is in $\langle a \rangle$.  Otherwise $a\circ \alpha=a\circ\alpha|_{\alpha^{-1}(\text{Dom}(a))}$.  Notice that $\text{Im}(\alpha)\cap \text{Dom}(a)$  is a subgraph of $H$.  Thus the domain $\alpha\circ a$ is a subgraph of $H$ and  $\alpha\circ a$ is an isomorphism of  subgraphs isomorphic to a subgraph of $H$.
\end{proof}

\begin{theorem}\label{thmminimal}
Given a graph $G$ and $a\in \Fisg(G)$ where $a:H\to K$, $\langle a \rangle=\bigcap I_i$ where $I_i$ is an ideal of $\Fisg(G)$ containing $a$.
\end{theorem}
\begin{proof}
Since $\langle a \rangle$ is an ideal containing $a$, we have that $\langle a \rangle\supseteq \bigcap I_i$.

To show the other containment, we let $I$ be an ideal containing $a$.  Let $L'$ be a  subgraph of $H$.  Then $id_{L'}\in \Fisg(G)$ and so $a\circ id_{L'}\in I$.  Since $\text{Im}({id_{L'}})$ is $L'$, we have that $a\circ id_{L'}=a|_{L'}$ and $a\circ id_{L'}$ is an isomorphism from $L'$ to an isomorphic  subgraph of $K$, let us call this subgraph $L''$.  Given any  subgraph of $G$ $M$ isomorphic to $L'$, there is an isomorphism  $\varphi:L'\to M$.  But notice that since $a\circ id_{L'}$ is an isomorphism, it is invertible and moreover $a\circ id_{L'}\in I$.  Thus $\varphi\circ (a\circ id_{L'})^{-1}\circ (a\circ id_{L'})\in I$ and $\varphi\in I$. 

Then consider an induced subgraph of $G$ $L$ isomorphic to $L'$.  There is then an isomorphism $\psi:L\to L'$.  Since $\text{Im}(\psi)=L'=\text{Dom}(\varphi)$, it follows that $\varphi\circ \psi$ is an isomorphism, where $\varphi\circ\psi:L \to M$.  Since $\varphi\in I, \varphi\circ\psi\in I$.

Thus, given any $L, M$ isomorphic to a  subgraph of $H$, there is an isomorphism  $\varphi\circ\psi\in I, \varphi\circ\psi:L\to M$.  So given any isomorphism $\gamma:L\to M$, $\gamma\circ (\varphi\circ\psi)^{-1}\circ (\varphi\circ\psi)\in I$ and $\gamma\in I$.  So it follows that $\langle a \rangle \subseteq I $.  Since $I$ was arbitrarily chosen, $\langle a \rangle\subseteq \bigcap I_i$ and so $\langle a \rangle =  \bigcap I_i$.
\end{proof}

\begin{corollary}
$\langle a \rangle = \Fisg(G)a\Fisg(G)$.
\end{corollary}

\begin{proof}
$\Fisg(G)a\Fisg(G)$ is the principle ideal containing $a$ \cite{SGT} and so by Theorem \ref{thmminimal} $\langle a \rangle\subseteq \Fisg(G)a\Fisg(G)$.  Conversely, given any $\phi \circ a\circ \rho\in \Fisg(G)a\Fisg(G)$, $\phi \circ a\in \langle a \rangle$ by the arguments in Proposition \ref{propisidealf}, and so by the same arguments, $(\phi \circ a)\circ \rho\in \langle a \rangle$.  Thus $\Fisg(G)a\Fisg(G)\subseteq \langle a \rangle$.
\end{proof}
In other words $\langle a \rangle$ is the minimal ideal of $\Fisg(G)$ containing $a$.

We notice that an ideal generated by a single element of $\Fisg(G)$ is totally determined by the domain subgraph of that element, and all of its subgraphs.  Since ideals are closed under unions and intersections, it is easy to see that any ideal is the union of all the principle ideals of its elements.  Thus one can see that an ideal is best understood by the subgraphs which constitute the domains of these functions.  In order to deal with these subgraphs in a clear way, we introduce the notion of a basis.

\begin{definition}
Let $G$ be a graph and let $I$ be an ideal of $\Fisg(G)$. Let $\mathcal{B}$ be a  family of graphs where given any $a\in I, a:H\to K$ $H$ is isomorphic to a subgraph of $B_i\in \mathcal{B}$.  Moreover, given any $B_i\in \mathcal{B}$, there is an $\alpha\in I$ where $B_i$ is the domain of $\alpha$.  We call $\mathcal{B}$ a \textit{generating set of graphs} of $I$.  If $\mathcal{B}$ is minimal, we say $\mathcal{B}$ is a basis of $I$.
\end{definition}

We can think of a basis of an ideal $I$ either as the minimal collection of subgraphs whose subgraphs are domains of elements of $I$, or we can think of the elements of $\mathcal{B}$ as the maximal subgraphs who are domains of elements of $I$.  We prove some essential properties of the basis.

\begin{lemma}\label{lembasisf}
Let $G$ be a graph and let $I$ be an ideal for $\Fisg{(G)}$.  Let $\mathcal{C}$ be a finite family of graphs which is a generating set of graphs for $I$.   Then there is a $\mathcal{B}\subseteq \mathcal{C}$ where $\mathcal{B}$ is a basis for $I$.
\end{lemma}
\begin{proof}
We begin with induction on $|\mathcal{C}|$.  If $|\mathcal{C}|=1$, then $\mathcal{C}=\{C_1\}$ is a singleton and is clearly minimal.  So we assume this is true for $\mathcal{C}$ where $|\mathcal{C}|<n$ and consider the case where $|\mathcal{C}|=n$.  If $\mathcal{C}$ is minimal with respect to being a generating set of graphs for $I$, then we let $\mathcal{B}=\mathcal{C}$ and we are done.  Otherwise there is a $C_j$, without loss of generality $C_n$ where $\mathcal{C}\backslash \{C_n\}$ remains a generating set of graphs for $I$.  By induction, there is a $\mathcal{B}\subseteq \mathcal{C}\backslash \{C_n\}$ minimal with respect to this property.  Thus $\mathcal{B}$ is a basis for $I$.
\end{proof}

\begin{corollary}
Given a graph $G$ and $I$ an ideal of $Fisg(G)$, $I$ has a basis $\mathcal{B}$.
\end{corollary}
\begin{proof}
We let $\mathcal{C}:=\{H\ |\ \exists a\in I, a:H\to K\}$.  Since $G$ is finite, $\mathcal{C}$ is finite, and by Lemma \ref{lembasisf} there is a basis for $I$, $\mathcal{B}\subseteq \mathcal{C}$.
\end{proof}

\begin{proposition}\label{propsubf}
Let $G$ be a graph and let $I, J$ be ideals of $\Fisg(G)$ where $I\subseteq J$.  Let $\mathcal{B}=\{B_1, \ldots B_m\}$ be a generating set for $J$.  Then there is a basis for $I$, $\mathcal{C}=\{C_1, \ldots, C_m\}$ where $C_i$ is isomorphic to a subgraph of some $B_j\in \mathcal{B}$.
\end{proposition}
\begin{proof}
Let $\mathcal{C}=\{C_1, \ldots C_m\}$ be a basis for $I$.  Consider $C_i$, there is an $a\in I, a:C_i\to K$.  Thus $C_i$ is the induced subgraph of $H'$ where $H'$ is a subgraph of $G$ and there is a $b\in J$ where $b:H'\to K'$.  Since $\mathcal{B}$ is a generating set for $J$, there is a $B_j$ where $H'$ is isomorphic to a  subgraph of $B_j$.  Thus $C_i$ is isomorphic to a  subgraph of $B_j$.
\end{proof}

\begin{proposition}\label{propunionf}
Given a graph $G$ and $I, J$ ideals of $\text{Iisg}(G)$ with basis $\mathcal{B}, \mathcal{C}$ respectively, then $\mathcal{B}\cup \mathcal{C}$ is a generating set of graphs for $I\cup J$.
\end{proposition}
\begin{proof}
Let $a\in I\cup J$, then $a\in I $ or $a\in J$.  If $a\in I$ then $a:H\to K$ where $H$ is isomorphic to a subgraph of $B_i\in \mathcal{B}$ otherwise if $a\in J$ then $a:M\to L$ where $M$ is isomorphic to a subgraph of $C_j\in \mathcal{C}$.  Either way, the domain of $a$ is isomorphic to a subgraph of an element of $\mathcal{B}\cup \mathcal{C}$.
\end{proof}

\begin{theorem}\label{thmuniquef}
Let $G$ be a graph, and let $I$ be an ideal of $\Fisg(G)$ with basis $\mathcal{B}$.  Then $\mathcal{B}$ is unique up to isomorphism.
\end{theorem}
\begin{proof}
Let $\mathcal{B}=\{B_1, \ldots B_m\}$ and $\mathcal{C}=\{C_1, \ldots C_n\}$ be basis of $I$.  Suppose that $B_m$ is not isomorphic to any $C_j\in \mathcal{C}$.  Then there must be a $C_k$ where $B_m$ is isomorphic to a subgraph of $C_k$.  Since $C_k\in \mathcal{C}$, $C_k$ is isomorphic to the domain of some $a\in I$.  Thus there is a $B_i$ where $C_k$ is isomorphic to a subgraph of $B_i$.  But then $B_m$ is isomorphic to a subgraph of $B_i$, contradicting the minimality of $\mathcal{B}$.
\end{proof}

\begin{proposition}
Let $H_1, \ldots H_n$ be a collection of subgraphs of $G$ where no $H_i$ is a subgraph of $H_j$ when $i\neq j$.  Then $\displaystyle\bigcup_{i=1}^n \langle id_{H_i} \rangle$ has basis $\mathcal{B}=\{H_1, \ldots, H_n\}$.
\end{proposition}
\begin{proof}
First, note that $\langle id_{H_i} \rangle$ contains a subgraph isomorphism whose domain is $H_i$.  Thus, $\langle id_{H_i} \rangle$ is the collection of all subgraph isomorphisms between any graphs isomorphic to a subgraph of $H_i$.  Let $I:=\displaystyle\bigcup_{i=1}^n \langle id_{H_i} \rangle$.  Notice that given any $a\in I$ it follows that $a\in \langle id_{H_i} \rangle$ for some $i$ and thus the domain of $a$ is isomorphic to a subgraph of $H_i$.  Thus $\mathcal{B}$ is a generating set for $I$.  But since no $H_i$ is a subgraph of any $H_j, i\neq j$ and each $id_{H_i}\in I$, it follows that $\mathcal{B}$ is minimal, since removing any $H_i$ would contradict $id_{H_i}\in I$.
\end{proof}
This collection of results show that each ideal is determined exactly by its basis, and exhibits a 1-1 correspondence between ideals of $\Fisg(G)$ and collections of subgraphs of $G$ where no element of these collections is isomorphic to a subgraph of another.

Finally, we consider the lattice structure of this ideal as in Theorems \ref{FisgLattice}, \ref{IisgLattice}, \ref{TisgLattice} and Proposition \ref{PisgSemilattice}.

\begin{theorem}\label{thmsemimod}
Let $G$ be a graph and let  $\mathcal{I}$ denote the collection of ideals of $\Fisg(G)$.  Then, $\mathcal{I}$ forms a distributive semimodular lattice under inclusion.
\end{theorem}
\begin{proof}
To show that $\mathcal{I}$ is a distributive lattice, we notice that by Lemma \ref{lemideal}, the elements of $\mathcal{I}$ are closed under union and intersection.  Thus it follows that $\mathcal{I}$ is a distributive lattice \cite{Birkhoff}, where given $A, B\in \mathcal{I}$, $A\wedge B=A\cap B$ and $A\vee B=A\cup B$.

To show semimodularity, notice Let $A, B\in \mathcal{I}$ such that $A\vee B$ does not cover $B$. That is, there is a $C\in \mathcal{I}$ such that $B<C<A\cup B$.  Thus there is a $\psi\in C\backslash B$ and since $C\subseteq A\cup B$ we have that $\psi\in A$.  But $\langle \psi\rangle \neq A$ or else $B\subseteq C, \langle \psi \rangle \subseteq C$ and $A\cup B \subseteq C$.

Thus consider $\langle \psi \rangle \cup (A\cap B)$.  By \ref{lemideal} this is an ideal of $\Fisg(G)$.  It is clear that $A\cap B\subseteq \langle \psi \rangle \cup (A\cap B)$, but moreover this containment is strict since $\psi\not\in B$.  Similarly $\langle \psi \rangle \cup (A\cap B)\subseteq A$ but this containment is strict, else if $\langle \psi \rangle \cup (A\cap B)=A$, then $(\langle \psi \rangle \cup (A\cap B))\cup B)=A\cup B$, but $\langle \psi \rangle \cup (A\cap B)\cup B=\langle \psi \rangle\cup B=C\neq A\cup B$.  Thus $A$ does not cover $A\cap B$ and $\mathcal{I}$ is semimodular.

\end{proof}

\begin{proposition}\label{propatomic}
Let $G$ be a graph other than $K_1$, and let $\mathcal{I}$ be the ideals of $\Fisg(G)$.  Then the lattice of $\mathcal{I}$ ordered by inclusion is not atomic.
\end{proposition}
\begin{proof}
Let $S$ denote the collection of proper subgraphs of $G$ and consider $X=\displaystyle\bigcup_{H\in S} \langle id_H \rangle$.  By Lemma \ref{lemideal} this is an ideal of $\Fisg(G)$ with a generating set of graphs $S$.  Let $\mathcal{B}\subseteq S$ be a basis for $X$.   So notice that $id_G\not \in X$, or else $G$ would be the subgraph of some element of $\mathcal{B}$, and all elements of $\mathcal{B}$ are proper subgraphs of $G$.  But clearly $id_G\in \langle id_G \rangle=\Fisg(G)$.  Since $G$ has proper subgraphs, $\langle id_G \rangle$ is not an atom, and since $\langle id_G \rangle$ is not the join of any other elements of $\mathcal{I}$, $\mathcal{I}$ is not atomic.
\end{proof}

Notice that the results of this section can be easily extended to $\Iisg(G)$, $\Tisg(G,v)$ and $\Pisg(G,v)$ with similar definitions.  In the case of $\Iisg(G)$, we replace the notion of subgraph with induced subgraph.  Since the induced subgraphs of induced subgraphs are in fact induced subgraphs of $G$, as are the intersection of such subgraphs, by mimicking the arguments of Proposition \ref{propisidealf} and Theorem \ref{thmminimal} we have:

\begin{theorem}
Let $G$ be a graph, $a\in \Iisg(G)$, consider the ideal generated by $a$, $\langle a \rangle:=\{\varphi:H\to K\}$ where $H, K$ are induced subgraphs of $G$ isomorphic to an induced subgraph of the domain of $a$.  Then $\langle a \rangle$ is an ideal of $\Iisg(G)$ and is the minimal such ideal containing $a$.
\end{theorem}
Similarly, following the arguments of Lemma \ref{lembasisf}, Propositions \ref{propsubf}, \ref{propunionf} and Theorem \ref{thmuniquef} we have:
\begin{theorem}
Let $G$ be a graph and $I$ an ideal of $\Iisg(G)$, and then let $\mathcal{B}$ be a collection of induced subgraphs of $G$ so that given any $a\in I$, the domain of $a$ is isomorphic to an induced subgraph of some $B_i\in \mathcal{B}$, and for each $B_i\in \mathcal{B}$, there is an $a\in I$ where the domain of $a$ is isomorphic to $B_i$.  We call $\mathcal{B}$ a generating set of graphs for $I$, and if $\mathcal{B}$ is minimal, then $\mathcal{B}$ is called a basis of $I$.

Then we have that each ideal $I$ of $\Iisg(G)$ has a basis, that basis is unique up to isomorphism, each generating set of graphs for an ideal $I$ contains a basis, the union of ideals $I, J$ has a generating set which is the union of the basis, and given any ideals $I, K$ where $I\subseteq K$, the elements of the basis for $I$ are induced subgraphs of the elements of the basis for $K$.
\end{theorem}
Finally, mimicking the arguments for Theorem \ref{thmsemimod} and Proposition \ref{propatomic}, we have that:
\begin{theorem}
Let $G$ be a graph and let $\mathcal{I}$ be the poset of ideals of $\Iisg(G)$ ordered by inclusion.  Then $\mathcal{I}$ is semimodular, but if $G$ contains more than 2 vertices, then $\mathcal{I}$ is not atomic.
\end{theorem}

The structure of $\Tisg(G,v)$ is nearly identical to the structure of $\Iisg(G)$ except that the domains of these subgraph isomorphisms are induced subgraphs which contain a specific root vertex $v$.  Nevertheless we may use similar arguments to obtain:
\begin{theorem}
Let $G$ be a graph, $v\in G$, $a\in \Tisg(G,v)$, consider the ideal generated by $a$, $\langle a \rangle:=\{\varphi:H\to K\}$ where $H, K$ are connected induced subgraphs of $G$ containing $v$ and isomorphic to a connected induced subgraph the domain of $a$ containing $v$.  Then $\langle a \rangle$ is an ideal of $\Tisg(G,v)$ and is the minimal such ideal containing $a$.
\end{theorem}

\begin{theorem}
Let $G$ be a graph, $v\in G$ and $I$ an ideal of $\Tisg(G,v)$, and then let $\mathcal{B}$ be a collection of connected induced subgraphs of $G$ containing $v$ so that given any $a\in I$, the domain of $a$ is  isomorphic to a connected induced subgraph of some $B_i\in\mathcal{B}$ containing $v$, and for each $B_i\in \mathcal{B}$, there is an $a\in I$ where the domain of $a$ is isomorphic to $B_i$.  We call $\mathcal{B}$ a generating set of graphs for $I$, and if $\mathcal{B}$ is minimal, then $\mathcal{B}$ is called a basis of $I$.

Then we have that each ideal $I$ of $\Tisg(G,v)$ has a basis, that basis is unique up to isomorphism, each generating set of graphs for an ideal $I$ contains a basis, the union of ideals $I, J$ has a generating set which is the union of the basis, and given any ideals $I, K$ where $I\subseteq K$, the elements of the basis for $I$ are connected induced subgraphs of the elements of the basis for $K$ and contain $v$.
\end{theorem}

\begin{theorem}
Let $G$ be a graph and let $\mathcal{I}$ be the poset of ideals of $\Tisg(G,v)$ ordered by inclusion.  Then $\mathcal{I}$ is semimodular, but if $G$ is connected and contains more than 2 vertices, then $\mathcal{I}$ is not atomic.
\end{theorem}

The structure of $\Pisg(G,v)$ is further complicated by the ordered graph, path structure.  However, analogous definitions and results still hold:
\begin{theorem}
Let $G$ be a graph, $v\in G$, $a\in \Pisg(G,v)$, consider the ideal generated by $a:(L, P_L)\to (M, P_M)$, $\langle a \rangle:=\{\varphi:(H, P_H)\to (K, P_K)\}$ where $P_H, P_K$ are paths rooted at $v$ and isomorphic to a subpath of $P_L$, $P_L'$, rooted at $v$, and $H, K$ are isomorphic to $G[V(P_L')]$.  Then $\langle a \rangle$ is an ideal of $\Pisg(G,v)$ and is the minimal such ideal containing $a$.
\end{theorem}

\begin{theorem}
Let $G$ be a graph, $v\in G$ and $I$ an ideal of $\Pisg(G,v)$, and then let $\mathcal{B}$ be a collection of pairs $(B_i, P_{B_i})$ where $P_{B_i}$ is a path rooted at $v$, $B_i=G[V(P_{B_i})]$ and given any $a\in I$, the domain of $a$ with domain $(H, P_H)$, we have that  $P_H$ is  isomorphic to a $v$ rooted subpath of some $P_{B_i}$ and $H$ is isomorphic to $G[V(P_{B_i})]$, and for each $(B_i, P_{B_i})\in \mathcal{B}$, there is an $a\in I$ where the domain of $a$ is isomorphic to $(B_i, P_{B_i})$.  We call $\mathcal{B}$ a generating set of graphs for $I$, and if $\mathcal{B}$ is minimal, then $\mathcal{B}$ is called a basis of $I$.

Then we have that each ideal $I$ of $\Pisg(G,v)$ has a basis, that basis is unique up to isomorphism, each generating set of graphs for an ideal $I$ contains a basis, the union of ideals $I, J$ has a generating set which is the union of the basis, and given any ideals $I, K$ where $I\subseteq K$, the elements of the basis for $I$ are $v$ rooted path induced subgraphs of the elements of the basis for $K$.
\end{theorem}

\begin{theorem}
Let $G$ be a graph and let $\mathcal{I}$ be the poset of ideals of $\Pisg(G,v)$ ordered by inclusion.  Then $\mathcal{I}$ is semimodular, but if $G$ is connected and contains more than 2 vertices, then $\mathcal{I}$ is not atomic.
\end{theorem}


\section{Graph Characterization by Inverse Semigroups}\label{Characterization}

We now consider the question of characterization. We show that $\Fisg(G)$ characterizes $G$ and conversely. We then show this characterization will not hold in $\Iisg(G)$ and give an infinite class of counterexamples.

\begin{theorem}\label{FISGRep}
Let $G$ and $H$ be graphs, then $\Fisg(G)\cong \Fisg(H)$ if and only if $G \cong H$.
\end{theorem}
\begin{proof}
$(\Rightarrow)$ Let $\Phi:\Fisg(G)\to  \Fisg(H)$ be a semigroup isomorphism. Let $v\in V(G)$, and $id_{v}:\{v\}\to \{v\}$ be the identity subgraph isomorphism. As $id_{v}$ is an idempotent $\Phi(id_{v})$ is an idempotent. By Lemma \ref{idemps}, $\Phi(id_{v})$ is a subgraph automorphism of $H$. Furthermore as the lattice structure of idempotents is preserved by the semigroup isomorphism, $\Phi(id_{v})$ corresponds to a single vertex subgraph of $H$ as $\Phi(id_v)$ covers the empty map. Thus we define $\phi:V(G)\to V(H)$ by $\phi(v)=\Phi(id_{v})(v)$. For $u\in V(H)$, we similarly define $\phi^{-1}:V(H)\to V(G)$ by $\phi^{-1}(u)=\Phi^{-1}(id_u)(u)$ and notice that $\phi^{-1}\circ\phi=id_{V(G)}$ and $\phi\circ\phi^{-1}=id_{V(H)}$. Hence $\phi$ is a bijection.\par 
Now let $e\in E(G)$ be a loop. Let $E$ be the subgraph of $e$ and its incident vertex, $v$. Then $\Phi(id_E)$ corresponds to an idempotent that only covers a single vertex, namely $\Phi(id_{v})$ as $\Phi(id_E \circ id_v)=\Phi(id_E)\circ\Phi(id_v)=\Phi(id_v)$. Thus $\Phi(id_E)$ is the identity of a subgraph of $H$ consisting of an edge and its incident vertex $\Phi(id_v)(v)$. Thus $\Phi(id_E)(e)$ is a loop in $E(H)$.\par
If $e\in E(G)$ is a non-loop edge and $E$ is the subgraph of $e$ and its two incident vertices $u$ and $v$, then $\Phi(id_E)$ covers the join of $\Phi(id_u)$ and $\Phi(id_v)$, and it does not cover any other single vertex isomorphisms. Hence $\Phi(id_E)$ is a subgraph consisting of an edge and two incident vertices, namely $\Phi(id_u)(u)$ and $\Phi(id_v)(v)$. Then $\Phi(id_E)(e)$ is an edge of $H$. Thus for $e\in E(G)$ we define $\theta:E(G)\to E(H)$ by $\theta(e)=\Phi(id_E)(e)$ for $E$ the subgraph consisting of $e$ and its incident vertices (or vertex if it is a loop). For $f\in E(H)$, we similarly define $\theta^{-1}:E(H)\to  E(G)$ by $\theta^{-1}(f)=\Phi^{-1}(id_F)(f)$ for $F$ the subgraph of $H$ consisting of $f$ and its incident vertices (or vertex if it is a loop). Then $\theta^{-1}\circ\theta=id_{E(G)}$ and $\theta\circ\theta^{-1}=id_{E(H)}$, and $\theta$ is a bijection.\par
Let $e\in E(G)$ incident to $u,v \in V(G)$. Then as $\Phi(id_E)(E)$ is a subgraph of $H$ with vertices $\Phi(id_u)(u)$ and $\Phi(id_v)(v)$, $\theta(e)=\Phi(id_E)(e)$ is incident to $\phi(u)=\Phi(id_u)(u)$ and $\phi(v)=\Phi(id_v)(v)$. Thus $\theta$ preserves incidence. Hence $\varphi=(\phi,\theta)$ is a graph isomorphism and $G\cong H$.\par
$(\Leftarrow)$ Let $\varphi:G \to H$ be an isomorphism. Then for $f:G_1\to G_2$ in $\Fisg(G)$ define $\Phi:\Fisg(G)\to Fisg(H)$ by $\Phi(f)=\varphi f \varphi^{-1}:\varphi(G_1)\to \varphi(G_2)$. As $\varphi$ is an isomorphism, $\varphi(G_1)\cong \varphi(G_2)$ in $H$ and $\Phi(f)\in \Fisg(H)$.\par
Now suppose $f,g\in \Fisg(G)$ with $f:G_1\to G_2$ and $g:G_3\to G_4$. Then $G_2 \cap G_3$ is a subgraph of $G$ and $g\circ f: f^{-1}(G_2\cap G_3)\to g(G_2\cap G_3)$ is an subgraph isomorphism. Hence $\Phi(g\circ f)=\varphi\circ(g\circ f) \circ \varphi^{-1}:\varphi(f^{-1}(G_2 \cap G_3))\to \varphi(g(G_2\cap G_3))$ and is an element of $\Fisg(G)$. Then, $\Phi(g\circ f)=\varphi\circ(g\circ f)\circ\varphi^{-1}=(\varphi\circ g \circ \varphi^{-1}) \circ (\varphi \circ f \circ \varphi^{-1})=\Phi(g)\circ\Phi(f)$. Hence $\Phi$ is a semigroup homomorphism.\par
Now for $j:H_1\to H_2$ in $\Fisg(H)$, we define $\Phi^{-1}:\Fisg(H)\to \Fisg(G)$ by $\Phi^{-1}(j)=\varphi^{-1}\circ j \circ \varphi:\varphi^{-1}(H_1)\to\varphi^{-1}(H_2)$. Similarly $\Phi^{-1}$ is a semigroup homomorphism and $\Phi^{-1}\circ\Phi=id_{\Fisg(G)}$ and $\Phi\circ\Phi^{-1}=id_{\Fisg(H)}$. Thus $\Fisg(G)\cong \Fisg(H)$.
\end{proof}

For $G$ the graph of a vertex with a loop and $K_1$ the single vertex graph, we have that $\Iisg(G)\cong \Iisg(K_1)$. Furthermore for $K_n$ and $\overline{K_n}$ the empty edge graph on $n$ vertices $\Iisg(K_n)\cong \Iisg(\overline{K_n})$. As isomorphisms preserve both adjacency and non-adjacency in simple graphs, an isomorphism of vertex induced subgraphs of a simple graph induces an isomorphism of the subgraphs of its complement on the same vertex sets. Thus they will have isomorphic induced subgraph inverse semigroups. We formally state this below.
\begin{proposition}
Let $G$ be a simple graph and $\overline{G}$ be its complement graph, then $\Iisg(G)\cong \Iisg(\overline{G})$.
\end{proposition}

\bibliographystyle{abbrv} 
\bibliography{GraphsandtheirISG}

\noindent T. Chih, Department of Sciences and Mathematics, Newberry College, Newberry, SC 29108, USA \\
\textit{E-mail}: \href{mailto:tien.chih@newberry.edu}{tien.chih@newberry.edu}\\ \\
\noindent D. Plessas, Department of Mathematics and Computer Science, Northeastern State University, Tahlequah, OK 74464, USA \\
\textit{E-mail}: \href{mailto:plessas@nsuok.edu}{plessas@nsuok.edu}\\

\end{document}